\documentclass[11pt]{article}

\usepackage{amsmath, amssymb, amsthm, verbatim,enumerate,bbm,color} 
\usepackage{indentfirst}
\usepackage{mathtools}
\usepackage[utf8]{inputenc}
\usepackage[english]{babel}
\usepackage{graphicx}
\usepackage{mathrsfs}
\usepackage{wrapfig}
\usepackage{float}
\usepackage{enumitem}
\usepackage{cleveref}

\title{Title}

\theoremstyle{plain}
\newtheorem{theorem}{Theorem}[section]
\newtheorem{lemma}[theorem]{Lemma}

\newtheorem{definition}{Definition}[section]
\newtheorem{corollary}[theorem]{Corollary}

\newcommand{\vol}{{\rm Vol}}

\newcommand{\dist}{{\rm dist}}

\newcommand{\R}{\mathbb{R}}

\RequirePackage[normalem]{ulem} 
\RequirePackage{color}\definecolor{RED}{rgb}{1,0,0}\definecolor{BLUE}{rgb}{0,0,1} 

\begin{document}
\title{On non-parallel cylinder packings}
\author{Ofek Eliyahu}
\maketitle

\begin{abstract}In this paper we will discuss optimal lower and upper density of non-parallel cylinder packings in $\mathbb R^3$ and similar problems. The main result of the paper is a proof of K. Kuperberg's conjecture for upper density (existence of a non-parallel cylinder packing with upper density ${\pi}/{\sqrt{12}}$).
Moreover, we prove that for every $\varepsilon>0$ there exists a non-parallel cylinder packing with lower density greater then $\frac{\pi}{6}-\varepsilon$.
\end{abstract}

\section{Introduction}
Denote by $B^n_r(x_0)\subset \mathbb R^n$ the ball of radius $r$  (with respect to the Euclidean metric) with center at a point $x_0$. 
Let $\mathcal C=\{C_{i}\}_{i=1}^{\infty}$ be congruent disjoint bodies in $\R^{n}$. We define 
\[
\delta^{+}({\mathcal C})=\limsup_{r\to\infty}\frac{\vol\big(B^n_r(0)\cap\bigcup_{i=1}^{\infty} (C_i)\big)}{\vol (B^n_r(0))},
\]
and
\[
\delta^{-}({\mathcal C})=\liminf_{r\to\infty}\frac{\vol\big(B^n_r(0)\cap\bigcup_{i=1}^{\infty} (C_i)\big)}{\vol(B^n_r(0))}.
\]
Also, provided the limit exists, we define  
\[
\delta({\mathcal C})=\lim_{r\to\infty}\frac{\vol\big(B^n_r(0)\cap\bigcup_{i=1}^{\infty} (C_i)\big)}{\vol(B^n_r(0))}.
\]

From now we consider only the values $n=3$  and $n=2$, and for $n=2$ we will denote $\vol$ by ${\rm Area}$. We say that a circle packing in $\R^{2}$ is a \textbf{lattice packing} if the centers of the circles form a lattice. In 1773, Joseph-Louis Lagrange proved that the highest-density lattice packing of circles is the hexagonal lattice. In 1942, L{\'a}szl{\'o} Fejes T{\'o}th proved that this packing is optimal among all circle packings (and that it's density  is $\frac{\pi}{\sqrt{12}}\approx0.9068$). A proof of that can be found in \cite{chang2010simple}.

Let $\ell$ be a line in $\R^3$, and let $r>0$. The  \textbf{infinite circular cylinder with axis $\ell$ and radius $r$} is the set of all points in $\R^{3}$ that lie at distance smaller than $r$ from $\ell$. Two infinite cylinders are said to be \textbf{parallel} if their axes are parallel. A collection $\mathcal C$ of infinite cylinders with the same radius is called a \textbf{cylinder packing} if the  interiors of the cylinders in $\mathcal C$ are pairwise disjoint. A cylinder packing is called a \textbf{non-parallel cylinder packing} if no two cylinders in it are parallel.
\smallskip

In 1989, A. Bezdek and W. Kuperberg, \cite{bezdek1990maximum}, proved that for every cylinder packing ${\mathcal C}=\{C_{i}\}_{i=1}^{\infty}$ one has $\delta^{+}({\mathcal C})\leq\frac{\pi}{\sqrt{12}}$. It is clear that this bound is tight among all cylinder packings. In 1990, K. Kuperberg, \cite{kuperberg1990nonparallel}, proved that there exists a non-parallel cylinder packing ${\mathcal C}=\{C_{i}\}_{i=1}^{\infty}$ with $\delta^{-}({\mathcal C})>0$  (for her packing, $\delta^{-}({\mathcal C})=\frac{\pi^{2}}{576}\approx0.0171$). She conjectured that the bound $\frac{\pi}{\sqrt{12}}$ is also tight among non-parallel cylinder packings, and that there exists a non-parallel cylinder packing with density $\frac{\pi}{\sqrt{12}}$ (she did not specify whether she meant lower or upper density).
 
One can define two more notions of density:
\[
(\delta^{*})^{+}({\mathcal C})=\limsup_{r\to \infty} \left( \sup_{x_0 \in \R^n }\frac{\vol\big(B^n_r(x_0)\cap\bigcup_{i=1}^{\infty} (C_i)\big)}{\vol(B^n_r(0))}\right)
\]
and
\[
(\delta^{*})^{-}({\mathcal C})=\liminf_{r\to \infty} \left(\inf_{x_0 \in \R^n}\frac{\vol\big(B^n_r(x_0)\cap\bigcup_{i=1}^{\infty} (C_i)\big)}{\vol(B^n_r(0))}\right)
\]
In some cases these two quantities are more natural than $\delta^{+}$ and $\delta^{-}$  because they are invariant under translations  and do not assign a special role to the origin. It is not hard to see that for every cylinder packing ${\mathcal C}=\{C_{i}\}_{i=1}^{\infty}$ one has that  $(\delta^{*})^{+}({\mathcal C})\geq\delta^{+}({\mathcal C})$ and $(\delta^{*})^{-}({\mathcal C})\leq\delta^{-}({\mathcal C})$.
But in fact one can do a reduction to the result of Bezdek and Kuperberg and show that $(\delta^{*})^{+}({\mathcal C})\leq\frac{\pi}{\sqrt{12}}$ for every cylinder packing ${\mathcal C}=\{C_{i}\}_{i=1}^{\infty}$.
\smallskip

In 1997, Claudia and Peter Graf, \cite{graf1997moglichst}, showed that there exists a non-parallel cylinder packing ${\mathcal C}=\{C_{i}\}_{i=1}^{\infty}$ with $\delta^{-}({\mathcal C})=\frac{5}{12}$ . 
In 2014, T. Hales  proved that the maximal density of sphere packings is $\frac{\pi}{3\sqrt{2}}\approx0.7404$. His proof used complex computer calculations, and solved the Kepler conjecture, which was one of the most challenging open problems on this topic \cite{hales2017formal}. In 2018, D. Ismailescu and P. Laskawiec, \cite{ismailescu2019dense}, showed that there exists a non-parallel cylinder packing ${\mathcal C}=\{C_{i}\}_{i=1}^{\infty}$  with $\delta^{-}({\mathcal C})=\frac{1}{2}$. They conjectured that for every non-parallel cylinder packing ${\mathcal C}=\{C_{i}\}_{i=1}^{\infty}$ there exists a sequence of holes with radius tending to infinity, i.e., there exist sequences $b_{n},r_{n}\rightarrow\infty$ as ${n\rightarrow\infty}$ such that for all $ n\in\mathbb{N}$ one has that $B(r_{n},b_n)\cap(\bigcup_{i=1}^{\infty} (C_i))=\emptyset$, which implies that $(\delta^{*})^{-}({\mathcal C})=0$ for every non-parallel cylinder packing ${\mathcal C}=\{C_{i}\}_{i=1}^{\infty}$.

In Section 2 of this paper, we   prove that for every $\varepsilon>0$ there exists a non-parallel cylinder packing ${\mathcal C}=\{C_{i}\}_{i=1}^{\infty}$ with $\delta^{-}({\mathcal C})\geq\frac{\pi}{6}-\varepsilon$ (\cref{first result}).
In Section 3, we show that there exists a non-parallel cylinder packing with $\delta^{+}({\mathcal C})=\frac{\pi}{\sqrt{12}}$ (in particular, $(\delta^{*})^{+}=\frac{\pi}{\sqrt{12}}$), i.e., a non-parallel cylinder packing with optimal upper density. This proves the aforementioned conjecture of K. Kuperberg. In fact we show a more general result, namely, we construct  a  non-parallel cylinder packing from a lattice circle packing, such that the upper density of the cylinder packing equals the density of the original circle packing. 

From now, we denote by O the origin, both in $\mathbb R^2$ and both in $\mathbb R^3$.

\section{Existence of a cylinder packing with higher lower density} 
\begin{theorem}
\label{prepare for first result}
For every $\varepsilon>0$, there are $K,L>0$ which satisfy $L^{2}=K^{2}+1$ and $r_{0}>0$ such that the following holds: Let   $A_1=(x_1,y_1), A_2=(x_2,y_2)\in \R^{2}$ be arbitrary points with integer Euclidean norms $
d_1=\|A_1\|$ and $d_2=\|A_2\|$ which satisfy the two  conditions
\[
d_1,d_2 \geq{r_{0}}
\]
and
\[
\widetilde{c}\geq{\frac{1+\varepsilon}{\max{(d_1,d_2)}}},
\]


\noindent where   $\widetilde{c}$ is the angle between the lines that pass through the points $O$ and $A_1$  and  $O$ and $A_2$,  respectively, i.e., $\widetilde{c}=\angle{A_{1}OA_{2}}$.

Then 
\[
\dist(\ell_{1},\ell_{2})\geq{1},
\]
where  $\ell_{1}=\{(x_{1},y_{1},0)+t(y_{1},-x_{1},Kd_{1}+L): t \in \mathbb{R}\}$, 
$\ell_{2}=\{(x_{2},y_{2},0)+t(y_{2},-x_{2},Kd_{2}+L): t \in \mathbb{R}\}$
\end{theorem}
 
\begin{proof}
Let $\varepsilon>0$. By Ismailescu and Laskawiec's calculation  in~\cite{ismailescu2019dense}, Lemma 3.4, we have that  
\newline
\newline

(1) \ \ \ \ $$\quad({\rm dist}(\ell_{1},\ell_{2}))^{2}=$$ \[\\
\begin{aligned}
&\frac{(1-c)^{2}d_{1}^{2}d_{2}^{2}\gamma^{2}+2L(1-c)d_{1}d_{2}\gamma(d_{2}-d_{1})^{2}+L^{2}(d_{2}-d_{1})^{4}}{-(1-c)^{2}d_{1}^{2}d_{2}^{2}+2(1-c)d_{1}d_{2}[L^{2}(1+d_{1}d_{2})+KL(d_{1}+d_{2})]+L^{2}(d_{2}-d_{1})^{2}},
\end{aligned}
\]
where $c=\cos\widetilde{c}$ and $\gamma = (Kd_{1}+Kd_{2}+2L)$.
The constants $L,K$ will be chosen to satisfy: 
\begin{enumerate}
\item \[K^{2}>\frac{1+0.25\varepsilon}{1+0.5\varepsilon}L^{2}\] \underline{} 
\item \[\forall h>0: \ \ \ \frac{1+2\varepsilon +\frac{\varepsilon^{2}}{2}}{2(1+h)}(2+h)^{2}K^{2}\geq{\frac{1+2\varepsilon}{2}\cdot{\frac{(2+h)^{2}}{1+h}}L^{2}}\]

\item 
\[K^{2}\geq{0.99L^{2}}\]
\end{enumerate}
which is clearly satisfied for $L$ large enough since $L^{2}=K^{2}+1$.
Then, the constant $r_{0}$ will be chosen such that for every $d_{1}>r_{0}$:
\begin{enumerate}[label=(\roman*)]
    \item \[
    1-\cos\left(\frac{1+\varepsilon}{d_{1}}\right)>\frac{1+0.25\varepsilon}{2d_{1}^{2}}\]

    by Taylor expansion of order 2 of the function $\cos{x}$, clearly this is satisfied for $r_{0}$ large   enough for all $d_{1}\geq{r_{0}}$.
    
    \item \[L^{2}+2KLd_{1}<0.25\varepsilon{L^{2}d_{1}^{2}}\] 
    \item \[\forall h>0\]  \[\left[(1+2\varepsilon)\left(2+\frac{h^{2}}{2(1+h)}\right)+4h^{2}-2h-2-\frac{2}{d_{1}^{2}}\right] \\ 
    {\geq}{3.999\varepsilon+4h^{2}-2h}\]
    
    \item \[
\frac{1}{100}[L^{2}\delta^{2}d_{1}^{2}(\delta^{2}d_{1}^{2}-1)]{\geq}{2.5L^{2}d_{1}^{2}}.
\] which clearly holds when $r_{0}$ is large enough for all $d_{1}\geq{r_{0}}$.
\end{enumerate}
We consider two separate cases.

\underline{Case 1:} $d_{1}=d_{2}$, note that in this case  \[
\begin{aligned}
\quad({\rm dist}(\ell_{1},\ell_{2}))^{2}=&\frac{(1-c)^{2}d_{1}^{4}\gamma^{2}}{-(1-c)^{2}d_{1}^{4}+2(1-c)d_{1}^{2}[L^{2}(1+d_{1}^{2})+2KLd_{1}]}\\
&\geq\frac{(1-c)^{2}d_{1}^{4}\gamma^{2}}{2(1-c)d_{1}^{2}[L^{2}(1+d_{1}^{2})+2KLd_{1}]}\\
&=\frac{(1-c)d_{1}^{2}\gamma^{2}}{2[L^{2}(1+d_{1}^{2})+2KLd_{1}]}\\
&\geq\frac{(1-\cos(\frac{1+\varepsilon}{d_{1}})d_{1}^{2}\gamma^{2}}{2[L^{2}(1+d_{1}^{2})+2KLd_{1}]},
\end{aligned}
\]
\[
\begin{aligned}
\underset{(i)}\geq \frac{{\frac{1+0.5\epsilon}{2d_1^2}}d_1^2\gamma^2}{2[L^2(1+d_1^2)+2KLd_1]},
\end{aligned}
\]
\[
\begin{aligned}
\underset{(1)}{\geq}&\frac{{\frac{1+0.25\epsilon}{2d_1^2}}d_1^2(2Ld_1)^2}{2[L^2(1+d_1^2)+2KLd_1]}\\
&\underset{(ii)}{\geq}\frac{2L^2d_1^2(1+0.25\epsilon)}{2L^2d_1^2(1+0.25\epsilon)}=1,
\end{aligned}
\]

which completes the proof for this case. \newline
\underline{Case 2:} $d_{1}\ne{d_{2}}$. From (1), an algebraic manipulation shows that  $\dist(\ell_{1},\ell_{2})\geq{1}$ if and only if $\varDelta\geq{0}$, where
\begin{align*}
&\varDelta = (1-c)^{2}d_{1}^{2}d_{2}^{2}\big[1+\gamma^{2}\big]+L^{2}(d_{2}-d_{1})^{2}\big[(d_{2}-d_{1})^{2}-1)\big]\\
&~~+2(1-c)d_{1}d_{2}\big[KL(d_{1}+d_{2})((d_{2}-d_{1})^{2}-1)+L^{2}(2d_1^{2}-5d_{1}d_{2}+2d_1^{2}-1)\big].    
\end{align*}
Denote 
\[
(2) \ \ \ \ \widetilde\varDelta=(1-c)d_{1}d_{2}(1+\gamma^{2})+2L^{2}(2d_{1}^{2}-5d_{1}d_{2}+2d_1^{2}-1)
\]
and
\[
\widetilde{\!\!\widetilde\varDelta}= \left[\frac{1}{(1-c)d_{1}d_{2}}L^{2}\big(d_{2}-d_{1}\big)^{2}+2KL\big(d_{2}+d_{1}\big)\right] \big[(d_{2}-d_{1})^{2}-1\big].
\]
Then $\varDelta=(1-c)d_{1}d_{2}(\widetilde\varDelta+\widetilde{\!\!\widetilde\varDelta})$. For every $d_1\ne{d_2}$ we have $(d_{2}-d_{1})^{2}-1\geq{0}$ (since $d_{1},d_{2}$ are integers), so 
$~\widetilde{\!\!\widetilde\varDelta}\geq{0}$. Consequently, if $\widetilde\varDelta\geq{0}$, then $\varDelta\geq{0}$.
\smallskip

Assume without loss of generality that $d_{1}<d_{2}$. Denote $d_{2}=(1+h)d_{1}$, where $h>0$. Then
\begin{align*}
    \widetilde\varDelta&\geq{\left[1-\cos\frac{1+\varepsilon}{(1+h)d_{1}}\right]}\cdot{(1+h)d_{1}^{2}\big(K(2+h)d_{1}\big)^{2}}\\
    &\qquad +2L^{2}\big[d_{1}^{2}(2+2(1+h)^{2}-5(1+h))-1\big]\\
    &\underset{\rm Taylor}{=}\frac{1}{2}{\left[\frac{1+\varepsilon}{(1+h)d_{1}}\right]^{2}\cos\widetilde{x}}\big(1+h)^{2}d_{1}^{2}((2+h)Kd_{1}\big)^{2}\\
    &\qquad +2L^{2}\big[d_{1}^{2}(2h^{2}-h-1)-1\big],  
\end{align*}
where  $0\leq{\widetilde{x}}\leq\displaystyle{{\frac{1+\varepsilon}{(1+h)d_{1}}}}$.
Note that $\widetilde{x}\rightarrow 0$ when $d_{1}\rightarrow \infty$,  and since $\lim_{x\to {0}}\cos{x}=1$, for $d_{1}$ large enough  we have $(1+\varepsilon)^{2}\cos{x}>\displaystyle{1+2\varepsilon +\frac{\varepsilon^{2}}{2}}$. It follows that  

\begin{align*}
    \widetilde\varDelta&\geq{\frac{1+2\varepsilon+\frac{\varepsilon^{2}}{2}}{2(1+h)^{2}d_{1}^{2}}}(1+h)d_{1}^{2}(2+h)^{2}d_{1}^{2}K^{2}+2L^{2}\big[d_{1}^{2}(2h^{2}-h-1)-1\big]\\
    &=\frac{1+2\varepsilon +\frac{\varepsilon^{2}}{2}}{2(1+h)}(2+h)^{2}d_{1}^{2}K^{2}+L^{2}d_{1}^{2}\left[2(2h^{2}-h-1)-\frac{2}{d_{1}^{2}}\right]\\
    &\underset{(2)}\geq{\frac{1+2\varepsilon}{2}\cdot{\frac{(2+h)^{2}}{1+h}}L^{2}d_{1}^{2}+L^{2}d_{1}^{2}\left[2(2h^{2}-h-1)-\frac{2}{d_{1}^{2}}\right]}\\
    &= L^{2}d_{1}^{2}\left[(1+2\varepsilon)\left(2+\frac{h^{2}}{2(1+h)}\right)+4h^{2}-2h-2-\frac{2}{d_{1}^{2}}\right]\\
    &\underset{ (iii)}{\geq}{L^{2}d_{1}^{2}(3.999\varepsilon+4h^{2}-2h)}.
\end{align*}
We see that for $L$ and $d_{1}$ large enough, (3)
$\widetilde\varDelta\geq{L^{2}d_{1}^{2}(3.999\varepsilon+4h^{2}-2h)}$.  
Now  let $\delta\geq{0}$ be such that for every $0\leq{h}\leq{\delta}$ we have  $4h^{2}-2h\geq{-\varepsilon}$.
Then, for every $0\leq{h}\leq{\delta}$ we get 
\[
\varDelta= (1-c)d_{1}d_{2}(\widetilde\varDelta+\widetilde{\widetilde\varDelta})\geq{(1-c)d_{1}d_{2}\widetilde\varDelta}\geq{(1-c)(3.99\varepsilon-\varepsilon)}\geq{0}.
\]
For $h\geq{\delta}$ we consider two  cases:
\begin{enumerate}
\item  If $(1-c)d_{1}d_{2}\leq{100}$, then 
\[
\widetilde{\widetilde\varDelta}\geq{\frac{1}{100}[L^{2}\delta^{2}d_{1}^{2}(\delta^{2}d_{1}^{2}-1)]}\underset{ (iv)}{\geq}{2.5L^{2}d_{1}^{2}}.
\] 
Since for every $h\geq{0}$ the inequality $4h^{2}-2h\geq{-2}$ holds, we get that
\begin{align*}
    \varDelta&=(1-c)d_{1}d_{2}(\widetilde\varDelta+\widetilde{\widetilde\varDelta})\geq{(1-c)d_{1}d_{2}L^{2}d_{1}^{2}(2.5+3.99\varepsilon +4h^{2}-2h)}\\
    &\geq{(1-c)d_{1}d_{2}L^{2}d_{1}^{2}(0.5+4h^{2}+3.99\varepsilon)}\geq{0}.
\end{align*}
\item If $(1-c)d_{1}d_{2}\geq{100}$, we again consider  two separate cases:
\begin{enumerate}
\item If $h\leq{1}$, then, by the defnition of $\widetilde\varDelta$  \ (2), we get that
\begin{align*}
\widetilde\varDelta\geq{100(1+\gamma^{2})-2L^{2}\cdot{10d_{1}^{2}}}\underset{ (3)}\geq{99L^{2}d_1^{2}-20L^{2}d_1^{2}}\geq{0},
\end{align*}
and so $\varDelta\geq{0}$.
\item If $h>1$, then $4h^{2}-2h\geq{2h}\geq{0}$, and using (3) we get that
\[
\varDelta\geq{(1-c)d_{1}d_{2}\widetilde\varDelta}\geq{(1-c)L^{2}d_{1}^{2}(3.99\varepsilon+4h^{2}-2h)}\geq{0}.
\]
\end{enumerate}
\end{enumerate}

Thus, for every $h>{0}$ we have $\varDelta\geq{0}$, which completes the proof. \end{proof}

\begin{theorem} For every $\varepsilon>0$, there exists a non-parallel cylinder packing ${\mathcal C}=\{C_{i}\}_{i=1}^{\infty}$ with lower density $\delta^{-}({\mathcal C})\geq{\frac{\pi}{6}-\varepsilon}$.
\label{first result}
\end{theorem}

Construction: Let $\varepsilon>{0}$, and let $L,K,r_{0}$ be three (large enough) constants  such that the statement of \cref{prepare for first result} holds for these values. For each $n\geq{r_{0}}$, let $k$ be the (unique) natural number such that $2^{k}\leq{n}<{2^{k+1}}$. We define

\[
{\mathcal A}_{n}=\Biggl\{ n\left( \cos{\frac{\left( 1+\varepsilon \right) \cdot{j}}{2^{k}}},\sin{\frac{\left( 1+\varepsilon \right) \cdot{j}}{2^{k}}},0 \right)   \\\  \Bigl| 1\leq{j}\leq{\frac{2\pi}{1+\varepsilon}\cdot{2^{k}}} \Biggl\}
\]

Now for every $n\geq{r_{0}}$  we draw from every point $A=(x,y,0) \in {\mathcal A}_{n}$ the line $\ell_{A}=(x,y
,0)+t(y,-x,Kn+L)$, and then  around each such line we take the cylinder of radius $\frac{1}{2}$. Let $C_{n}$ be the set of all the cylinders corresponding to $n$. For every 2 cylinders as above (which might correspond to the same $n$ or to distinct values of $n$), if they come from 2 points which lie on a straight line with the origin, the axes lie in 2 parallel planes which are perpendicular to the xy-plane and the distance between them is at least 1, so the distance between the axes is at least 1. Else, for every $n_{1}\geq{n_{0}}$, for $i=0,1$ let $k_{i}$ be the (unique) natural number such that $2^{k_{i}}\leq{n_{i}}<{2^{k_{i}+1}}$. Let $A_{0}\in {\mathcal A}_{0}$, $A_{1} \in {\mathcal A}_{1}$, and assume that $A_{0}$, $A_{1}$, O not lie on a straight line. Clearly that $\angle{A_{0}OA_{1}}$ is of the form $\frac{(1+\varepsilon)\cdot{m}}{2^{k_1}}$ for an integer $m$. Since $2^{k_{1}}\leq{\|A_1\|}$ the assumptions of \cref{prepare for first result} are satisfied. Then, by \cref{prepare for first result}, the distance between the axes is at least 1. So the cylinders are disjoint and it's a cylinder packing.  Denote the resulting cylinder packing by $\mathcal C$ (the union about all $n\geq{r_{0}}$). That is, $\mathcal C = \bigcup_{n\geq{r_{0}}} {C_{n}}$.

\begin{definition}
Let ${\mathcal C}=\{C_{i}\}_{i=1}^{\infty}$ be a cylinder packing such that the axes of all the cylinders in $C_i$ intersect the $(x,y)$-plane. The {\bf dual circle packing} of ${\mathcal C}$, denoted $\widetilde{{\mathcal C}}$, is the circle packing in the plane that consists of all the circles that are centered at the points where the axes of the cylinders $C_i$ intersect the $(x,y)$-plane and each circle has the same radius as the corresponding cylinder.
\end{definition}

\begin{definition}
Let $C$ be a cylinder with axis that intersects the $(x,y)$-plane, but not contained in this plane. The {\bf dual cylinder} of $C$, denoted $C^*$, has the same radius as $C$ and it's axis is the line that passes through the point where the axis of $C$ intersects the $(x,y)$-plane and is perpendicular to the $(x,y)$-plane.
\end{definition}

\begin{lemma}
\label{dual circle packing density}
Let ${\mathcal C}=\{C_{i}\}_{i=1}^{\infty}$ be a cylinder packing with the property that for every $i$ the axis of $C_i$ intersects the 
$(x,y)$-plane in a point $A_i$ and is perpendicular to the line $OA_i$. Let $\widetilde{{\mathcal C}}$ be the dual circle packing of ${\mathcal C}$. Then $\delta^{+}({\mathcal C})=\delta^{+}(\widetilde{{\mathcal C}})$ and  $\delta^{-}({\mathcal C})=\delta^{-}(\widetilde{{\mathcal C}})$. 
\end{lemma}

\begin{proof} If the axes of all the cylinders $C_i$ are perpendicular to the $(x,y)$-plane, the lemma follows from Fubini's theorem - For every plane of the form z=a, the intersection of the plane with a ball which centered at the origin is a circle, and the intersection of the circle with the cylinders is a lift of the intersection of $\widetilde{{\mathcal C}}$ with the projection of the circle to the $(x,y)$-plane, then we can use the density of $\widetilde{{\mathcal C}}$ in all circles expect in 2 domes, which are negligible for big radius of the ball. In the general case,  we claim that for every cylinder $ C_{i}$ and every $r>0$, 
\[
\vol(C_{i}\cap{B^3_{r}(0)})=\vol(C^*_{i}\cap{B^3_{r}(0)}),
\]
which clearly will complete the proof.

To prove the claim, we note that the axis of the dual  cylinder $C^*_{i}$ is the image of the axis of $C_{i}$ under a rotation-transformation $T$ around $OA_i$. Since $T$ is an isometry,  $d(x,\ell)=d(T(x),T(\ell))$   for every line $\ell$,   in particular for $\ell_{A_i}$. Denote $\ell=\ell_{A_i}$ for simplicity. Therefore, if $d(x,\ell)<r$, then $d(T(x),T(\ell)<r$, and so if $d(y,T(\ell))<r$ we have
\[
d(T^{-1}(y),T^{-1}(T(\ell)))\underset{T ~{\rm is}~1-1}{=}d((T^{-1}(y),\ell)<r.
\]
Hence, $C^*_{i}=T(C_{i})$, and since $O$ lies on $OA_i$, necessarily $T(O)=O$. Moreover,  since $T$ is an isometry,   $\|(T(x))\|= \| x \|$ for all $x\in\mathbb R^3$, so $C^*_{i}\cap{B^3_r(0)}=T(C_{i}\cap{B^3_r(0)})$. We conclude that $\vol( C^*_{i}\cap{B^3_r(0)})=\vol(T(C_{i}\cap{B^3_r(0)})$, which completes the proof in this case.
\end{proof}

\begin{proof}[Proof of \cref{first result}]

Let us calculate the lower density of the dual circle packing $\widetilde {\mathcal C}$ in the plane which consists of the circles with radius $1/2$ centered at the points in $\bigcup_{j=r_{0}}^{\infty}({\mathcal A}_{j})$. By \cref{dual circle packing density},  $\delta^{+}({\mathcal C})=\delta^{+}(\widetilde{\mathcal C})$ and $\delta^{-}(\mathcal C)=\delta^{-}(\widetilde{\mathcal C})$. Since in every annulus of the form $D_{r} = {B^2_{r+1}(0)}/B^2_r(0)$ for $2^{n}< r \leq{2^{n+1}-1}$ there is the same number of centers of circles of the circle packing and $vol(D_{r})$ is monotonically increasing in $r$, if the limit  
\[
\lim_{n\to\infty}\frac{{\rm Area}(B^2_{2^{n}}(0)\cap \widetilde {\mathcal C})}{{\rm Area}(B^2_{2^{n}}(0))}
\] 
exists, then it is equal to the lim\,inf.

Since for $2^{n}< k \leq{2^{n+1}}$ the density of centers (number of centers of circles of the circle packing divided by the area) in annulus of the form $D_{k} = {B^2_{k}(0)}/B^2_{2^{n}}(0)$ is decreasing as a function of k, it's easy to show that the density of centers in $B^2_{k}(0)$ is at least the minimum of the densities in $B^2_{2^{n}}(0)$ and $B^2_{2^{n+1}}(0)$, so we can't get a lower partial limit.

Let  ${\mathcal U}_n$ denote the union of the circles in  $\widetilde{\mathcal C}$ with center of norm $r$ such that $2^n\le r<2^{n+1}$. Then  since
\begin{align*}
    &\lim_{n\to\infty}\frac{{\rm Area} ({\mathcal U}_n)}{{\rm Area}(B^2_{2^{n+1}}(0))-{\rm Area}(B^2_{2^{n}}(0))}\\
    &\qquad\qquad\qquad= \lim_{n\to\infty}\frac{\frac{\pi}{4}\cdot{2^{n}}\cdot{\frac{2}{1+\varepsilon}}\cdot{2^{n}}}{3\pi\cdot{2^{2n}}}\\
    &\qquad\qquad\qquad=\frac{\pi}{6(1+\varepsilon)},
\end{align*}
Stolz's theorem implies that:
\begin{enumerate}[label=(\roman*)]
\item \[
\lim_{n\to\infty}\frac{{\rm Area}(B^2_{2^{n}}(0)\cap \widetilde {\mathcal C})}{{\rm Area}(B^2_{2^{n}}(0))}=\frac{\pi}{6(1+\varepsilon)}.
\]
\end{enumerate}
It follows that $\delta^{-}(\mathcal C)=\frac{\pi}{6(1+\varepsilon)}{\rightarrow}\frac{\pi}{6}$ as $\varepsilon\to 0$, which completes the proof.   
\bigskip
\end{proof}
\begin{theorem}
\label{calc upper density}
The upper density of the previous construction is $\delta^{+}(C)= \frac{3\pi}{16(1+\varepsilon)}.$
\end{theorem}
\begin{proof}

let $c$,  $0\leq{c}\leq{1}$, be an arbitrary  constant and consider the subsequence $n_{k}=2^{k}(1+c)$. Then, using (i) we get that
\[
\begin{aligned}
&\qquad\qquad\qquad \lim_{n\to\infty}\frac{{\rm Area}(B^2_{n_{k}}(0)\cap \widetilde {\mathcal C})}{{\rm Area}(B^2_{n_{k}}(0))} \\
&=\lim_{n\to\infty}\frac{{\rm Area}(B^2_{2^{k}}(0)\cap \widetilde {\mathcal C})+{\rm Area}(B^2_{n_{k}}(0))\cap \widetilde {\mathcal C}/B^2_{2^{k}}(0))}{{\rm Area}(B^2_{n_{k}}(0))} \\
&=\frac{\frac{\pi}{6(1+\varepsilon)}\cdot{{\rm Area}(B^2_{2^{k}}(0))}(1+o(1))+2^{k}\cdot{c}\cdot{2^{k}}\cdot{\frac{\pi}{4}}\cdot{\frac{2\pi}{1+\varepsilon}}}{\pi(2^{k}(1+c))^{2}} \\
&\qquad\qquad\qquad\qquad=\frac{\pi\cdot(1+3c)}{6(1+\varepsilon)(1+c)^{2}} .
\end{aligned}
\]
\end{proof}
\begin{lemma}
\label{calculus 1 subsequences lemma}
Let $\mathcal \{a_n\}_{n=1}^{\infty}$ be a sequence,
and let $A$ be a set of convergent subsequences of $\mathcal \{a_n\}_{n=1}^{\infty}$
which cover the sequence. Assume that for every $\varepsilon>0$ there exists $M>0$ such that for every $n>M$ and every subsequence in A which includes ${a_n}$, $|{a_n}-l|<\varepsilon$ where l is the limit of the subsequence. Let B be the set of partial limits of subsequences in $A$, then B has a maximum and
\begin{align*}
    \limsup_{n\to\infty} {a_{n}} = \max {B}
\end{align*}
\end{lemma}
\begin{proof}
the proof is easy and it's an exercise to the reader.
\end{proof}

Since  for every $\varepsilon>0$ there exists a $k_{0}\in{\mathbb{N}}$, uniform in $c$, such that 
\[
\left|\frac{{\rm Area}(B^2_{n_{k}}(0)\cap \widetilde {\mathcal C})}{{\rm Area}(B^2_{n_{k}}(0))}-\frac{\pi\cdot(1+3c)}{6(1+\varepsilon)(1+c)^{2}}\right|<\varepsilon,
\]
for every 
$k\geq{k_{0}}$, we can use \cref{calculus 1 subsequences lemma} and conclude that
\begin{align*}
    \delta^{+}(C)&=\limsup_{r\to\infty}\frac{\vol(B^3_r(0)\cap \widetilde {\mathcal C})}{\vol (B^3_r(0))}\\
    &=\max_{0\leq{c}\leq{1}}\frac{\pi\cdot(1+3c)}{6(1+\varepsilon)(1+c)^{2}}\\
    &=\frac{3\pi}{16(1+\varepsilon)} 
\end{align*}
(In order to calculate the maximum we looked at the derivative and set it equal to 0).
\begin{corollary}
Let $\mathcal C=\{C_{i}\}_{i=1}^{\infty}$ be a cylinder packing with the property that for every $i$ the axis of $C_i$ intersects the 
$(x,y)$-plane in a point $A_i$ such that the axis is perpendicular to the line $OA_i$, and the property that the set of the intersection points of the axes with the $(x,y)-plane$ satisfy:
\begin{enumerate}
\item They have integer norm.
\item The distance between any two of the points is at least 1.

\end{enumerate}
Then, $\delta^{+}({\mathcal C})\leq {\frac{3\pi}{16}}$.
\end{corollary}
\begin{proof}
The proof is very similar to the proof of \cref{calc upper density}, and it's an exercise to the reader.
\end{proof}

\section{Existence of a non-parallel cylinder packing with upper density $\frac{\pi}{\sqrt{12}}$}
Consider the following  two sequences, defined recursively:
\begin{enumerate}
\item $a_{1}=1$, $a_{k+1}=100a_{k}$.
\smallskip
\item $T_{1}= 1$, $T_{k+1}=2^{10a_{k+1}}\cdot{T_{k}}$.
\end{enumerate}

For every $k\in\mathbb{N}$, define $\Omega_{k}=\{x\in \mathbb R^{2} |\, 2^{a_{k}} < {\| x \|} \leq {2^{2a_{k}}}\}$, and for every $x=(x_{1},y_{1},0)\in{\Omega_{k}}$, consider the straight line $\ell_{x}=\{(x_{1},y_{1},0)+t(y_{1},-x_{1},T_{k}): t \in \mathbb{R}\}$.
\begin{lemma}
\label{not parallel axises}
For any pair of distinct points $p_{1},p_{2}\in\bigcup_{k=1}^{\infty}\Omega_{k}$, the lines $\ell_{1}=\ell_{p_{1}}$ and $\ell_{2}=\ell_{p_{2}}$ are not parallel.
\end{lemma}

\begin{proof}
Assume, by contradiction, that $\ell_{1}$ and $\ell_{2}$ are parallel. Denote $p_{1} = (x_{1}, y_{1})$ and $p_{2} = (x_{2}, y_{2})$ and assume that $p_{1}\in{\Omega_{k_{1}}}$ and $p_{2}\in{\Omega_{k_{2}}}$, $k_{2}\geq{k_{1}}$. Then, there exists $ {c}\in\mathbb{R}$ such that $(y_{2},-x_{2},T_{k_{2}})=c(y_{1},-x_{1},T_{k_{1}})$.

If $k_{1}=k_{2}$, then $c=1$ and so $(x_{1},y_{1})=(x_{2},y_{2})$, a contradiction.

Now suppose that $k_1\neq k_2$. Then 
\[
\frac{T_{k_{2}}}{T_{k_{1}}}=c=\frac{\|(x_{2},y_{2}) \|}{\| (x_{1},y_{1}) \|}.\tag{$\ast$}
\]
For every $k\in\mathbb{N}$, and points $p_{1}\in{\Omega_{k}}$, $p_{2}\in{\Omega_{k+1}}$, $\frac{T_{k+1}}{T_{k}}=2^{1000a_{k}}$ and 
\[
\frac{\| (x_{2},y_{2}) \|}{\| (x_{1},y_{1}) \|}\leq{\|(x_{2},y_{2}) \|}\leq{2^{200a_{k}}}<2^{1000a_{k}}=\frac{T_{k+1}}{T_{k}}.
\] 
Hence, by induction, for any two points $x_{1}\in{\Omega_{k_{1}}}$ and   $x_{2}\in{\Omega_{k_{2}}}$ ($k_{2}>k_{1}$)  we have
\[
\frac{\| (x_{2},y_{2}) \|}{\| (x_{1},y_{1}) \|}<\frac{T_{k_{2}}}{T_{k_{1}}},
\]
which contradicts the equality in $(\ast)$. Therefore, $\ell_{1}$ and $\ell_{2}$ are not parallel.
\end{proof}

\begin{theorem}
\label{distance between axises}
There exists $k_{0}\in\mathbb{N}$ such that for every $\varepsilon>0$ there exists $m_{0}\in\mathbb{N}$ such that for any pair  $k_{2}>m_{0}$, $k_{1}>k_{0}$, $k_{2} > k_{1}$ and any pair of points  $A_{1}\in \Omega_{k_{1}}$ and $A_{2}\in \Omega_{k_{2}}$ for which \ \ \ \ 
 (1) \ \ \ \  $c = \cos{\theta}$ where $\theta$ is angle between the line that passes  through the points O and $A_{1}$ and the line that passes through O and $A_{2}$ satisfies $|c|\geq{\frac{1}{d_{2}^{0.975}}}$. Then  
\[
\ \ \ \ {\rm dist}(\ell_{1},\ell_{2})\geq{(1-\varepsilon)d_{1}}
\]
 $($in particular, ${\rm dist}(\ell_{1},\ell_{2})\geq{(1-\varepsilon)})$, where $d_{1}=\| A_{1} \|$, $d_{2}=\| A_{2} \|$, $\ell_{1}=\ell_{A_{1}}$, $\ell_{2}=\ell_{A_{2}}$ .
\end{theorem}

\begin{proof}[Proof of \cref{distance between axises}]
Denote $T'_{1}=T_{k_{1}}$, $T'_{2}=T_{k_{2}}$ and let $v_{1},v_{2}$ be the direction vectors of $\ell_{1},\ell_{2}$, respectively.
We have   
\[
{\rm dist}(\ell_{1},\ell_{2})=\frac{\big|\overset{\longrightarrow}{A_1A_2}\cdot(v_1\times v_2)\big|}{\|v_1\times v_2\|}
\]
where $\cdot$ denotes the inner product..
We calculate:
\begin{align*}
\big|\overset{\longrightarrow}{A_1A_2}&\cdot(v_1\times v_2)\big|=
\left|\begin{bmatrix}
x_2-x_1 & y_2-y_1 & 0 \\
y_1 & -x_1 & T_1' \\
y_2 & -x_2 & T_2'
\end{bmatrix}\right|\\
&=\, -T'_1\left|\begin{bmatrix}
x_2-x_1 & y_2-y_1 \\
y_2 & -x_2
\end{bmatrix}\right|+T'_2\left|\begin{bmatrix}
x_2-x_1 & y_2-y_1 \\
y_1 & -x_1
\end{bmatrix}\right|\\
&= T'_1\big(x_2^2+y_2^2-(x_1x_2+y_1y_2)\big)+T'_2\big(d_1^2-(x_1x_2+y_1y_2)\big).
\end{align*}
Note  that $x_1x_2+y_1y_2=cd_1d_2$. Hence, by algebraic manipulation,
\begin{equation*}
\big|\overset{\longrightarrow}{A_1A_2}\cdot(v_1\times v_2)\big|=T'_1\big(d_2^2-cd_1d_2)+T'_2(d_1^2-cd_1d_2\big).
\end{equation*}
Next,
\begin{align*}
\|v_1&\times v_2\|^2 =\|v_1\|^2\|v_2\|^2-(v_{1} \cdot {v_{2}}) ^2
=(d_1^2+{T'_1}^2)(d_2^2+{T'_2}^2)-(x_1 x_2+y_1 y_2+T'_1 T'_2)^2\\
&=(d_1^2+{T'_1}^2)(d_2^2+{T'_2}^{2})-(c d_1 d_2+T'_1 T'_2)^2\\
&=(1-c^2)d_2^2(d_1^2+{T'_1}^2)+(d_1T'_2-c T'_1d_2)^2.
\end{align*}

Therefore,
\begin{equation*}
\text{dist}^2(\ell_{1},\ell_{2})=\frac{(d_1^2T'_2+d_2^2T'_1-(T'_1+T'_2)d_1d_2c)^2}{(1-c^2)d_2^2(d_1^2+{T'_1}^2)+(d_1T'_2-cT'_1d_2)^2},
\end{equation*}
so the inequality
\[
\text{dist}^2(\ell_{1},\ell_{2})\geq (1-\varepsilon)d_1^2 
\]
which leads the desired result, is equivalent to
\[
(2) \ \ \ \   \big(d_1^2T_2'+d_2^2T'_1-(T_1+T_2)d_1d_2c\big)^2\geq(1-\varepsilon)d_1^2(1-c^2)d_2^2(d_1^2+{T'_1}^2)+(d_1T'_2-cT'_1d_2)^2.
\]
Denoting $d_2=d_1+h$,  the left-hand side of (2) takes the form 
\begin{align*}
&\big(d_1^2T_2'+d_2^2T'_1-(T_1'+T'_2)d_1d_2c\big)^2
\\&=\big(d_1^2T'_2+(d_1^{2}+2hd_1+h^2)T'_1-(T_1'+T_2')d_1d_2c\big)^2\\
&=\big(T_1'((1-c)d_1^{2}+(2-c)hd_1+h^{2})+T_2'(d_1^{2}(1-c)-hd_1c))
\big)^2.
\end{align*}
We want to verify that
\begin{equation*}
|d_1^{2}(1-c)-hd_1c|\geq d_1^2.
\end{equation*}
it suffices to show that $|hd_1c| \geq{|3d_1^{2}|}$, which is equivalent to
\begin{equation*}
 |c|\geq \frac{3d_1}{h}.
\end{equation*}
For $d_1$ large enough (since $h>d_1^{50}$) its enough that $|c|\notin(0,\frac{1}{(d_1+h)^{0.975}})$, as in our assumptions.
Denote 
\[
I=\big|T_1'\big((1-c)d_1^{2}+(2-c)hd_1+h^{2}\big)\big|.
\]

If $I<T_2d_1^{2}$, then  
\begin{equation*}
\big(T_1'((1-c)d_1^{2}+(2-c)hd_1+h^{2})+T_2'(d_1^{2}(1-c)-hd_1c))
\big)^2 \geq(T_2'd_1^{2}-I)^2.
\end{equation*}

\begin{align*}
I\leq 2hd_1T'_1+h^2T'_1+T'_1d_1^2\leq &d_2^2T'_1+d_2^2T'_1+2d_2^2T'_1
\\
&=4d_2^2T'_1\leq 4\cdot 2^{4a_{k_2}}T'_1=2^{10a_{k_2}}T'_1\frac{4}{2^{6a_{k_2}}}
\end{align*}
\begin{equation*}
\frac{T'_2}{T'_1}\geq\frac{T_{a_{k_2}}}{T_{a_{k_2-1}}}=2^{10a_{k_2}}\Rightarrow2^{10a_{k_2}}T_1\leq T'_2
\end{equation*}
\begin{equation*}
\Rightarrow I\leq T'_2\cdot o(1)\Rightarrow I\leq T_2'd_1^2\cdot o(1)
\end{equation*}
when the o(1) term is as a function of $k_{2}$, as $k_{2}$ tend to $\infty$. Then (for $k_2$ large enough):
\begin{equation*}
\big(T_1'((1-c)d_1^{2}+(2-c)hd_1+h^{2})+T_2'(d_1^{2}(1-c)-hd_1c))
\big)^2\geq (T_2d_1^{2})^2(1+o(1))
\end{equation*}
and for $k_2$ large enough 
\begin{equation*}
\geq (1-\frac{\varepsilon}{2})(T'_2d_1^2)^2. 
\end{equation*}
and since $d_{2}T'_{1} = o(d_{1}T'_{2})$, on the right side of (2): 

\begin{align*}
&(1-\varepsilon)d_1^2((1-c^2)d_2^2(d_1^2+{T'_1}^2)+(d_1T'_2-cT'_1d_2)^2)\leq(1-\varepsilon)(d_1^2(T_2'd_1)^2\cdot o(1)+(d_1T'_2)^2\cdot(1+o(1)))\leq
\\&{(1-0.75\varepsilon)(T'_2d_1^2)^{2}(1+o(1))}
\end{align*}
when the o(1) term is as a function of $k_{2}$, as $k_{2}$ tend to $\infty$. And again for $k_2$ large enough the inequality holds, then
\begin{equation*}
\text{dist}(\ell_{1},\ell_{2})^2\geq(1-\varepsilon)d_1^2
\end{equation*}
\begin{equation*}
\Rightarrow \dist(\ell_{1},\ell_{2})\geq\sqrt{1-\varepsilon}d_1\geq(1-\varepsilon)d_1
\end{equation*}
\end{proof}

\begin{definition} If $\mathcal L$  is a lattice in $\mathbf{R}^{2}$, define 
\[
{\rm density}(\mathcal{L})=\lim_{r\to\infty}\frac{|B_{r}^{2}(0)\cap \mathcal{L}|}{\pi\cdot{r^{2}}}
\]
$($for lattices the limit always exists$)$.
\end{definition}

\begin{theorem}
\label{result 2}
Let $\mathscr{L}$ be a lattice in $\mathbb R^{2}$ with the property that  $d(x_{1},x_{2})\geq{1}$  for any two points $x_{1},x_{2}\in\mathscr{L}$ . Then, for every $\varepsilon>0$ there exists a non-parallel cylinder packing  $\mathcal C=\{C_{i}\}_{i=1}^{\infty}$  such that:
\begin{enumerate}
\item[{\rm 1.}] For any cylinder in $\mathcal C$, the intersection of its axis with the $(x,y)$-plane lies in $\mathscr{L}$.
\item[{\rm 2.}] The upper density of $\mathcal C$   is at least ${\rm density}(\mathscr{L})\cdot{\frac{\pi}{4}}\cdot{(1-\varepsilon)}$.
\end{enumerate}
\end{theorem}

The proof requires the following Lemma that show an intuitive  idea- for every lattice $\mathscr{L}$ and partition of circle of radius $r$ into cut sectors with same angle and area tends to $\infty$, the distribution of the sector of the points on the lattice which is in the circle, is converging to uniform(i.e, every sector has almost the same number of points on the lattice which is in the circle)
\begin{lemma}
Let $\mathscr{L}$ be a lattice in $\mathbb R^{2}$ and S is a convex bounded subset of $R^{2}$, such that the boundary of S is in the class Lip(n,M,L) (a definition of Lip(n,M,L) can be found in \cite{widmer2011}, Definition 2.2). Then $| |\mathscr{L}\cap{S}| - {\rm density}(\mathscr{L})\cdot{area(S)}| \leq{C(n, L, \mathscr{L})\cdot{M}}$  where $C(n, L, \mathscr{L})$ is a constant that depend only on the constants $n, L$ and the lattice $\mathscr{L}$.
\begin{proof}
    This follows directly from \cite{widmer2011}, Theorem 2.4
\end{proof}
\end{lemma}
\begin{lemma}
\label{LEM:UniformDistributionOfLattice}
Let $\mathscr{L}$ be a lattice in $\mathbb {R^{2}}$ that is spanned over $\mathbb {Z}$ by two linearly independent vectors. Let $r>0$ and let $\theta_{1}$ and $\theta_{2}$ be numbers such that 
$0\leq{\theta_{1}}<\theta_{2}\leq{2\pi}$
and  $\theta_{2}-\theta_{1}\geq{\frac{1}{2\sqrt{r}}}$. Then in $\mathscr{L}\cap{B^2_r(0)}$ there exist  
\[
{\rm density}(\mathscr{L})\cdot{\pi}\cdot{r^{2}}\cdot{\frac{\theta_{2}-\theta_{1}}{2\pi}}\cdot{(1+o(1))}
\]
points such that the angle $\alpha(x)$ between the line that passes through such a point and the origin and the positive $x$-axis lies in $[\theta_{1},\theta_{2}]$. The term $o(1)$ is bounded by a factor, g(r) such that g(r)=o(1) when r tends to infinity, and g(r) depends only on $r$ and $\mathscr  L$ and not on the specific choice of $\theta_{1},\theta_{2}$. 
\end{lemma}
\begin{proof}
It's easy to check that every section of circle with radius r (WLOG $r\in{\mathcal{N}}$) is in the class Lip(2,10r, 1). Then, by the previous lemma, in every section C with angle $\alpha$ of circle with radius r, there are
\[
{\rm density}(\mathscr{L})\cdot{\pi}\cdot{r^{2}}\cdot{\alpha}+ q
\] 
points in $\mathscr{L}\cap{C}$, where $|q|$ bounded by $C(\mathscr{L})\cdot{10r}$. which finishes the proof.
\end{proof}
\begin{theorem}
\label{distance between axises outside a ball}
For every $\varepsilon>0$ there exist an $M>0$ such that the following holds: for any $k_{1},k_{2}>M$ and any two points $x_{1}\in \Omega_{k_{1}}$ and $x_{2}\in \Omega_{k_{2}}$ such that $k_{1}=k_{2}$ and $\|x_{1}-x_{2} \|\geq{1}$ or the pair $x_{1},x_{2}$  satisfies condition (1) of \cref{distance between axises}, 
\[
{\rm dist}(\ell_{x_{1}},\ell_{x_{2}})\geq{1-\varepsilon}.
\]
\end{theorem}
\begin{lemma}
\label{Ismailescu lemma}
{\rm (Ismailescu and Laskawiec, \cite{ismailescu2019dense})}. Let $r,R>0$, $A_{1}=(x_{1},y_{1},0)$, $A_{2}=(x_{2},y_{2},0)$. Assume  $\| A_{1}A_{2} \|\geq{2r}$, $\| A_{1}\|\leq{R}$ and $\| A_{2}\|\leq{R}$.   Let $\ell_{i}$ denote the straight line that passes through $A_{i}$ and has the direction vector $v_{i}=(y_{i},-x_{i},T)$ where $8r^{2}T\geq{R^{4}}$, $i=1,2$. Then 
\[
{\rm dist}(\ell_{1},\ell_{2})\geq 2r\Big(1-\frac{1}{T}\Big).
\]
\end{lemma}
A proof of this lemma  can be found in \cite{ismailescu2019dense}, Lemma 2.1. 

\begin{proof}[Proof of \cref{distance between axises outside a ball}] By \cref{distance between axises}, there exist $k_{0}$ such that for every $k_{1}, k_{2} > k_{0}$, $k_{1}\neq{k_{2}}$, and every $x_{1}\in \Omega  _{k_{1}}$, $x_{2}\in{\Omega_{k_{2}}}$ which satisfies (1), the assertion is true.  Therefore, we need to find $M'=M'(\varepsilon)$ such that the assertion holds true for every $k>M'$ and $x_{1}\neq{x_{2}}\in \Omega_{k}$. 
\smallskip

Now, for every $k\in\mathbb{N}$ and every point $x\in \Omega_{k}$, $\| x \|\leq{2^{2a_{k}}}$, and for every point $y\in \Omega_{k}$, $y\neq x$, with $\|x-y \|\geq{1}=2\cdot{\frac{1}{2}}$, choose $r=\frac{1}{2}$ and $R=2^{2a_{k}}$. Then $8r^{2}T_{k}=2T_{k}\geq{2^{100a_{k}}}\geq{2^{8a_{k}}}=R^{4}$. Hence, by \cref{Ismailescu lemma},  ${\rm dist}(\ell_{1},\ell_{2})\geq{2\cdot{\frac{1}{2}}\cdot{(1-\frac{1}{T_{k}})}}=1-\frac{1}{T_{k}}$ and since $T_{k}\to\infty$ as $k\to\infty$, there exists $M'=M'(\varepsilon)$, such that for every $k>M'$, $1-\frac{1}{T_{k}}\geq{1-\varepsilon}$. Then, this implies that ${\rm dist}(\ell_{1},\ell_{2})\geq{1-\varepsilon}$.
\end{proof}

\begin{proof}[Proof of \cref{result 2}]  Construction: Let $\varepsilon>0$, and let $M=M(\varepsilon)$ be the constant whose existence is asserted in \cref{distance between axises outside a ball}. For every $k\geq{M}$ define the set $\widetilde{\Omega_{k}}$  recursively by
\begin{enumerate}
\item $\widetilde{\Omega_{ M }}= \Omega_{M} \cap \mathscr{L}$.
\item For every $k>M$, $\widetilde{\Omega_{k}}$ is the set of all points in $\Omega_{k}\cap \mathscr{L}$ that satisfy (1) with all the points in $\bigcup_{i=M}^{k-1} \widetilde{\Omega_{ i }}$.
\end{enumerate}

 By \cref{LEM:UniformDistributionOfLattice} and the Taylor expansion of the cosine function around $\frac{\pi}{2}$, for every $x\in\bigcup_{i=M}^{k-1} \widetilde{\Omega_{i}}$ there exist at most $b\cdot{\frac{1}{2^{\cdot{0.975}{a_{k}}}}}\cdot{(2^{2a_{k}})^{2}}$ points in $\Omega_{k} \cap \mathscr{L}$ such that the absolute value of the cosine of the angle between the line that passes through the point and the origin and the line that passes through $x$ and the origin lies in the range  $I=[0,\frac{1}{2^{0.975a_{k}}})$ where b is a constant (clearly that all the others satisfy (1) with x). Since  $\bigcup _{i=M}^{k-1} \widetilde{\Omega_{i}}\subseteq{B_{2^{2a_{k-1}}}(0)}$, there exist M', such that for every $k>M'$ there exist at most 
 \begin{align*}
     |\bigcup_{i=M}^{k-1} \widetilde{\Omega_{ i }}|&\leq{|B_{2^{2a_{k-1}}}(0)\cap \mathscr{L}|}\\
     &\leq{({\rm density}(\mathscr{L})+0.01)\cdot{\pi}\cdot{(2^{2a_{k-1}})^2}}\leq{\left(\frac{4}{\sqrt{12}}+0.01\right)\cdot{\pi}\cdot{2^{4a_{k-1}}}}\\
     & \leq{10\cdot{2^{4a_{k-1}}}}
 \end{align*}
 points in $\bigcup_{i=M}^{k-1} \widetilde{\Omega_{ i }}$. ($\frac{4}{\sqrt{12}}$ is the maximal density of a lattice in the plane which satisfy that the distance between every 2 points on the lattice is at least 1, see \cite{chang2010simple}).
 
Then for every $k>\max{\{M,M'\}}$ there are at most
\[
b\cdot\frac{1}{2^{0.975a_{k}}}\cdot{(2^{2a_{k}})^{2}}\cdot{10}\cdot{2^{4a_{k-1}}}=10b\cdot{2^{3.065a_{k}}}
\]
points in $\Omega_{k} \cap \mathscr{L}$  which don't satisfy (1) with some point in $\bigcup_{i=M}^{k-1} \widetilde{\Omega_{i}}$. We define $\widetilde{\Omega_{k}}$ to be the set of all points in $\Omega_{k} \cap \mathscr{L}$ without these points. Now for every point in $\widetilde{\Omega_{k}}$ we will construct a cylinder with axis $(x,y,0)+t(y,-x,T_{k})$ and radius $\frac{1}{2}\cdot{(1-\varepsilon)}$. By \cref{distance between axises}, it is a cylinder packing, and by \cref{not parallel axises} it is non-parallel cylinder packing. denote the cylinder packing by $C$. Since
\[
\lim_{k\to\infty}\frac{10b\cdot{2^{3.065a_{k}}}}{{\rm Area}\big(B_{2^{2a_{k}}}^{2}(0)\big)}=0,
\]
the points we remove will not affect  the density of the dual circle packing $\widetilde{\mathcal C}$. Then, by \cref{dual circle packing density}, it is readily concluded that 

$\delta^{+}(\mathcal C)\geq{\frac{\pi}{4}\cdot{(1-\varepsilon)^{2}}\cdot{density(\mathscr{L})}}$, which finishes the proof.
\end{proof}

\begin{theorem}
 For every lattice $\mathcal{L}$, there exist a non-parallel cylinder packing $\mathcal C=\{C_{i}\}_{i=1}^{\infty}$  such that:
\begin{enumerate}
\item[{\rm 1.}] For any cylinder in $\mathcal C$, the intersection of its axis with the $(x,y)$-plane lies in $\mathcal{L}$.
\item[{\rm 2.}] The upper density of $\mathcal C$   is at least ${\rm density}(\mathcal{L})\cdot{\frac{\pi}{4}}$.
\end{enumerate}
\end{theorem} 
\begin{proof}
Let $\varepsilon_{n}=\frac{1}{n}$, for every $n\geq{10}$. Let $k_{0}$, $t_{n} = M'(\varepsilon_{n})$ be the constants whose existence is asserted in \cref{distance between axises} (WLOG, $t_{n}$ is monotonic increasing and $t_{10}\geq{k_{0}}$). Let $p_{n}$ be the constant whose existence is asserted in \cref{distance between axises outside a ball}. Define $m_{n} = max(\{p_{n}, t_{n}\})$.
Now, for every $k\geq{m_{10}}$ define the set $\widetilde{\Omega_{k}}$  recursively by
\begin{enumerate}
\item $\widetilde{\Omega_{m_{10}}}= \Omega_{m_{10}} \cap \mathscr{L}$.
\item For every $k>m_{10}$, $\widetilde{\Omega_{k}}$ is the set of all points in $\Omega_{k}\cap \mathscr{L}$ that satisfy (1) with all the points in $\bigcup_{i=m_{10}}^{k-1} \widetilde{\Omega_{ i }}$.
\end{enumerate}

By the same calculation of \cref{result 2}, the points we remove don't affect about the density of the set of the points on the $(x,y)$-plane that we choose.
For every $k\geq{m_{10}}$, $(x_{1},y_{1},0)\in{\widetilde{\Omega_{ k }}}$, we take the axis $\ell_{x}=\{(x_{1},y_{1},0)+t(y_{1},-x_{1},T_{k}): t \in \mathbb{R}\}$.
Now, for every $n\geq{10}$, and for every $k_{2}\geq{k_{1}}\geq{m_{10}}$ such that $k_{2}\geq{m_{n}}$, for every $x_{1}\in{\widetilde{\Omega_{ k_{1} }}}$, $x_{2}\in{\widetilde{\Omega_{ k_{2} }}}$ define $\ell_{1} = \ell_{x1}$, $\ell_{2} = \ell_{x2}$
\begin{enumerate}
\item If $k_{1}=k_{2}$, then ${\rm dist}(\ell_{1},\ell_{2})\geq{1-\varepsilon_{n}}$
\item If $k_{2}\neq{k_{1}}$  (without loss of generality we can assume that $k_{2}>k_{1}$) and if $x_{1},x_{2}$ satisfy (1), then ${\rm dist}(\ell_{1},\ell_{2})\geq{(1-\varepsilon_{n})d_{1}}$, when $d_{1}=\| x_{1} \|$.
\end{enumerate}
Now, take $m_{10},m_{11}$ which are suitable for $\varepsilon_{10}$, $\varepsilon_{11}$ respectively (without loss of generality, we can assume that $m_{11}>m_{10}$). For every $m_{10}\leq{k}<m_{11}$, and for every $x\in \Omega_{k}$, take $x'=\frac{1}{1-\varepsilon_{10}}\cdot{x}$,  and let $\ell_{x'}$ be the line that passes through $x'$ and has the same direction vector as $\ell_{x}$. Note that $\ell_{x'}=\frac{1}{1-\varepsilon_{10}}\cdot{\ell_{x}}$. Then, take $m_{11},m_{12}$ which are suitable for $\varepsilon_{11}$, $\varepsilon_{12}$ respectively (without loss of generality, we can assume that $m_{12}>m_{11}$). For every $m_{11}\leq{k}<m_{12}$, and for every $x\in \Omega_{k}$ take $x'=\frac{1}{1-\varepsilon_{11}}\cdot{x}$, let $\ell_{x'}$ be the line that passes through $x'$ and has the same direction vector as $\ell_{x}$. Again, note that that $\ell_{x'}=\frac{1}{1-\varepsilon_{11}}\cdot{\ell_{x}}$. Then take $m_{12},m_{13}$ which are suitable for $\varepsilon_{12}$, $\varepsilon_{13}$, and so on. 
Again, we don't affect the density of the set of the points that we choose.

Now consider the collection $\mathcal C$ of all the cylinders of radius $\frac{1}{2}$ and with axes of the form $\ell_{x'}$. We claim that this collection is  a cylinder packing

Indeed, 
\begin{enumerate}
\item If $x_{1},x_{2}\in{\Omega_{k}}$ for some $k\in\mathbb{N}$, then,  by \cref{distance between axises outside a ball}, ${\rm dist}(\ell_{x'_1},\ell_{x'_2})\geq{1}$.
\item If $x_1\in \Omega_{k_1}$ and $x_2\in \Omega_{k_2} $ with $k_1\neq k_2$, then
\[ 
\begin{aligned}
&
{\rm dist}(\ell_{x'_1},\ell_{x'_2})={\rm dist}(\frac{1}{1-\varepsilon_{k_1}}\ell_{x_1},\frac{1}{1-\varepsilon_{k_2}}\ell_{x_2})
\\
&
={\rm dist}
(\frac{1}{1-\varepsilon_{k_2}}\cdot \frac{1-\varepsilon_{k_2}}{1-\varepsilon_{k_1}}\ell_{x_1},\frac{1}{1-\varepsilon_{k_2}}\ell_{x_2})
=
\frac{1}{1-\varepsilon_{k_2}}{\rm dist}(\frac{1-\varepsilon_{k_2}}{1-\varepsilon_{k_1}}\ell_{x_1},\ell_{x_2})
\\
&
\geq{\rm dist}(\ell_{x_1},\ell_{x_2})-\left|\frac{1-\varepsilon_{k_2}}{1-\varepsilon_{k_1}}-1\right|d_1 \geq  0.9d_1-0.2d_1 =0.7d_1\geq 1.
\end{aligned}
\]
\end{enumerate}
So  the collection ${\mathcal C}$ is a cylinder packing, and by \cref{not parallel axises} it is non-parallel cylinder packing. By \cref{dual circle packing density} it's upper density is at least $\frac{\pi}{4}\cdot{\rm density}(\mathscr{L}) $.  
\end{proof}

\begin{corollary}
 There exist a non-parallel cylinder packing $\mathcal C=\{C_{i}\}_{i=1}^{\infty}$ with $\delta^{+}({\mathcal C})={\frac{\pi}{\sqrt{12}}}$.
\end{corollary}

\begin{proof}
If we take $\mathscr{L}$ to be the hexagonal lattice, then  we get that there exist a cylinder packing C such that on the one hand, $\delta^{+}({\mathcal C})\geq{\frac{\pi}{\sqrt{12}}}$, and on the other hand (see \cite{bezdek1990maximum}),  $\delta^{+}({\mathcal C})\leq{\frac{\pi}{\sqrt{12}}}$. Thus, $\delta^{+}({\mathcal C})={\frac{\pi}{\sqrt{12}}}$, as desired.
\end{proof}

\section{Conclusions and directions for future study}
In this paper we showed that the maximal value of $\delta^{+}$ and $(\delta^{*})^{+}$ for non-parallel cylinder packing is $\frac{\pi}{\sqrt{12}}$, and for every $\varepsilon>0$ there exist a non-parallel cylinder packing, $\mathcal C=\{C_{i}\}_{i=1}^{\infty}$ with $\delta^{-}(\mathcal C)>\frac{\pi}{6}-\varepsilon$.

One interesting question is what is the maximal value of $\delta^{-}$ for non-parallel cylinder packing(or supremum if maximum does not exist) and whether we can achieve the bound of $\frac{\pi}{\sqrt{12}}$.
Another interesting question is whether a non-parallel cylinder packing  $\mathcal C=\{C_{i}\}_{i=1}^{\infty}$ exists, with $(\delta^{*})^{-}(\mathcal C)>0$.\newline 

\textbf{Acknowledgments}\newline
This paper is based on my M.Sc  thesis at Tel Aviv University under the supervision of Prof. Barak Weiss.

I am deeply grateful to Prof. Barak Weiss, whose experience and guidance have  enriched me. His insights and advice were indispensable for this work. 

I would also like to thank to Andrei Lacob for dedicating his time to review this work and provide helpful comments.

This research was supported by the Israel Science Foundation 2919/19.\newline

\bibliographystyle{unsrt}
\bibliography{name.bib}

\end{document}